\documentclass[english,11pt]{amsart}
\usepackage{latexsym}
\usepackage{multirow}
\usepackage{amsmath}
\usepackage{amssymb}
\usepackage{amsfonts}
\usepackage[T1]{fontenc}
\usepackage{amsthm}
\usepackage{enumerate}
\usepackage{babel}
\usepackage{tikz}
\usetikzlibrary{matrix,arrows}
\usepackage{tikz-cd}
\usepackage{cite}

\usepackage{hyperref}

\newtheorem{theorem}{Theorem}
\newtheorem{proposition}[theorem]{Proposition}
\newtheorem{question}[theorem]{Question}
\newtheorem{corollary}[theorem]{Corollary}
\theoremstyle{definition}
\newtheorem{example}[theorem]{Example}
\newtheorem{construction}[theorem]{Construction}
\theoremstyle{remark}
\newtheorem{remark}[theorem]{Remark}

 \newcommand{\im}{\operatorname{im}}
\newcommand{\rank}{\operatorname{rank}}
\newcommand{\xra}[1]{\xrightarrow{#1}}

\newcommand{\Seg}{\operatorname{Seg}}

\newcommand{\rr}{\operatorname{r}}
\newcommand{\crr}{\operatorname{cr}}
\newcommand{\length}{\operatorname{length}}

\newcommand{\hook}{{\: \lrcorner \:}}

\begin{document}

\author{Maciej Ga\l{}\k{a}zka}
\address{Maciej Ga\l{}\k{a}zka, Faculty of Mathematics, Computer Science, and Mechanics, University of Warsaw, ul. Banacha 2, 02-097 Warszawa, Poland}
\email{mgalazka@mimuw.edu.pl}
\title{Vector bundles give equations of cactus varieties}
\date{\today}

\begin{abstract}
  We prove that vector bundles give equations of cactus varieties. We derive from it that equations coming from vector bundles are not enough to
  define secant varieties of Veronese varieties in general.
\end{abstract}

\maketitle

\section{Introduction}

Suppose $W$ is a vector space over an algebraically closed field $\mathbb{K}$. We denote by $W^*$ the dual vector space. Let $X \subseteq \mathbb{P} W$ be a
non-degenerate (i.e.\ not contained in a hyperplane) projective variety over $\mathbb{K}$. For $F \in W$ let us define the notion of $X$-rank.
\begin{equation*}
  \rr_X(F) = \min \{r \in \mathbb{Z}_{\geq 0} | [F] \in \langle p_1,\dots, p_r\rangle \text{ for some } p_1,\dots,p_r \in X \}\text{,}
\end{equation*}
where $[F]$ denotes the class of $F$ in the projective space, and $\langle \cdot \rangle$ denotes the (projective) linear span. The $r$-th secant variety is
\begin{align*}
  \sigma_r(X) &= \overline {\{ [F] \in \mathbb{P}W | \rr_X(F) \leq r \}} \\
  &= \overline{\bigcup_{p_1,\dots,p_r \in X} \langle p_1,\dots, p_r\rangle}\text{,}
\end{align*}
where overline denotes the Zariski closure. The variety $X$ is often fixed, so we omit $X$ in the $X$-rank and $\rr_X(F)$, and write simply rank and
$\rr(F)$.

In this article we investigate the problem of finding set-theoretic equations of $\sigma_r(X)$ and the problem of giving lower bounds for rank. The
following proposition, which is given for instance in \cite[beginning of Chapter 7]{landsberg_tensorbook}, is useful:
\begin{proposition}\label{proposition:catalecticant}
  Suppose we have a linear map $j : W \to A\otimes B$, where $A, B$ are finite dimensional vector spaces. Let $k$ be a positive
  integer such that for every $p \in X$ we have
  \begin{equation*}
    \rank(j(p)) \leq k\text{.}
  \end{equation*}
  Then for any $F \in W$ we have
  \begin{equation*}
    \rr(F) \geq \frac {\rank j(F)} k \text{.}
  \end{equation*}
  There is an equivalent way of formulating the conclusion: suppose $r$ is a positive integer, then 
  \begin{equation*}
    j(\sigma_r(X)) \subseteq \sigma_{rk}(\Seg(\mathbb{P}A\times \mathbb{P}B)) \text{.}
  \end{equation*}
  Here $\Seg(\mathbb{P}A\times \mathbb{P}B) \subseteq \mathbb{P}(A\otimes B)$ is the image of the Segre embedding. Hence the $(rk + 1)$-th minors of
  $j \in A\otimes B \otimes W^*$, interpreted as a matrix with entries in $W^*$, give equations for $\sigma_r(X)$.
\end{proposition}
The problem is to find linear maps $j : W \to A\otimes B$ which give useful equations. In \cite{landsberg_ottaviani_equations} two methods are given.
First we follow the construction in \cite[Point 1.3 and Section 5]{landsberg_ottaviani_equations}.

\begin{construction}\label{construction:vector_bundle_version}
  Suppose the embedding $X \subseteq \mathbb{P}W$ is given by the complete linear system $H^0(X, \mathcal{L})$ of a very ample line bundle
  $\mathcal{L}$. This gives $W \cong H^0(X, \mathcal{L})^*$. Let $\mathcal{E}$ be a locally free sheaf of finite rank on $X$. Consider the canonical
  morphism $\mathcal{E}\otimes \mathcal{L} \otimes \mathcal{E}^\vee \to \mathcal{L}\text{,}$ where $\mathcal{E}^\vee$ denotes the dual sheaf. This
  yields
  \begin{equation*}
    H^0(X, \mathcal{E}) \otimes H^0(X, \mathcal{L}\otimes \mathcal{E}^\vee) \to H^0(X, \mathcal{L})\text{,}
  \end{equation*}
  which, after rearranging the factors, gives the map 
  \begin{equation}\label{equation:j_apolarity}
    j = j^\mathcal{E} \colon H^0(X, \mathcal{L})^* \to H^0(X, \mathcal{E})^* \otimes H^0(X, \mathcal{L}\otimes \mathcal{E}^\vee)^*\text{.}
  \end{equation}
  In this case we have $j(x) \leq \rank \mathcal{E}$ for every $x \in X$, see \cite[Proposition 2.20]{teitler_geometric_lower_bounds} or
  \cite[Proposition 5.1.1]{landsberg_ottaviani_equations}. Hence Proposition \ref{proposition:catalecticant} gives a lower bound for rank.
\end{construction}

Similarly, define the notion of cactus $X$-rank.
\begin{equation*}
  \crr(F) = \min \{r \in \mathbb{Z}_{\geq 0} | [F] \in \langle R \rangle\text{, } R \hookrightarrow X \text{ zero-dimensional, } \length R \leq r \}
  \text{.}
\end{equation*}
We denote by $\length R$ the length of a zero-dimensional scheme $R$, i.e.\ $h^0(X, R)$. Here $\langle \cdot \rangle$ denotes the (projective) linear span of a
scheme.

The $r$-th cactus variety is
\begin{align*}
  \kappa_r(X) &= \overline {\{ [F] \in \mathbb{P}(H^0(X,\mathcal{L})^*) | \crr(F) \leq r \}} \\
  &= \overline{\bigcup_{R \hookrightarrow X\text{, } \length R \leq r} \langle R\rangle}\text{.}
\end{align*}
For reasons to study cactus variety and cactus rank, see \cite[Subsection 1.3]{mgalazka_multigraded_apolarity}, or \cite[Subsections 1.2 and
1.3]{nisiabu_jabu_cactus}. In \cite{nisiabu_jabu_cactus}, using the idea of cactus varieties, equations for many secant varieties were found. Another
motivation is Corollary \ref{corollary:not_enough_equations}.

It is an interesting question when Proposition \ref{proposition:catalecticant} works for cactus rank and cactus variety, i.e.\ for what $X$, $j : W
\to A\otimes B$, the following equivalent statements
\begin{gather*}
  \crr(F) \geq \frac {\rank j(F)} k \text{ for all } F \in W\text{,}\\
  j(\kappa_r(X)) \subseteq \sigma_{rk}(\Seg(\mathbb{P}A\times \mathbb{P}B)) \text{ for all } r \text{ natural}
\end{gather*}
are true. Here $k$ is defined as in Proposition \ref{proposition:catalecticant}.

The following example shows that Proposition \ref{proposition:catalecticant} does not work for cactus rank in general.

\begin{example}[\cite{jabu_private}]
  Suppose $A, B$ are $5$-dimensional vector spaces. Consider the Segre embedding
  \begin{equation*}
    \mathbb{P}A \times \mathbb{P}B \hookrightarrow \mathbb{P}(A \otimes B)\text{.}
  \end{equation*}
  Let $X = \sigma_2(\mathbb{P}A \times \mathbb{P} B)$, so that $X$ is the projectivization of the set of matrices of rank at most $2$. The singular
  points of $X$ are the matrices of rank $1$. Take $j : A\otimes B \to A\otimes B$ to be an isomorphism. The highest rank of $j(F)$ for $[F] \in X$ is
  $2$. Hence Proposition \ref{proposition:catalecticant} for $r = 2$ says that the determinant of $A\otimes B$ gives is a non-trivial equation of
  $\sigma_2(X)$.  But $\kappa_2(X)$ is $\mathbb{P}(A\otimes B)$. This follows from the fact that the projectivized tangent space at any singular point
  of $X$ is $\mathbb{P}(A\otimes B)$.
\end{example}

In this paper, we show that when we choose $j$ as in Construction \ref{construction:vector_bundle_version}, then Proposition
\ref{proposition:catalecticant} works for cactus variety and cactus rank.
\begin{theorem}
  \label{theorem:vector_bundle_bound}
  For any $F \in H^0(X, \mathcal{L})^*$, and any locally free sheaf $\mathcal{E}$ on $X$ of rank $e$, we get
  \begin{equation*}
    \crr(F) \geq \frac{\rank{j(F)}}{e}\text{.}
  \end{equation*}
  In other words, if we fix bases of $H^0(X, \mathcal{E})$ and $H^0(X, \mathcal{L}\otimes \mathcal{E}^\vee)^*$, then for any $r \geq 0$ the $(re +
  1)$-th minors of $j$ (which is a matrix with entries in $H^0(X, \mathcal{L})$) give equations for $\kappa_r(X)$.
\end{theorem}
We prove Theorem \ref{theorem:vector_bundle_bound} in Section \ref{section:proof}.

As a corollary of this Theorem and main theorems of \cite{nisiabu_jabu_cactus}, we get that equations coming from vector bundles are not enough to
define secant varieties of the Veronese embedding in general. To be more precise:
\begin{corollary}\label{corollary:not_enough_equations}
  Let $v_d : \mathbb{P} V \to \mathbb{P} S^d V$ be the $d$-th Veronese embedding of the $n$-dimensional projective space $\mathbb{P} V$. Let $r$ be a
  positive integer. For each vector bundle $\mathcal{E}$ on $\mathbb{P} V$ of rank $e$, let $Z_{r,\mathcal{E}} \subseteq \mathbb{P}S^d V$ be the
  vanishing set of equations coming from $(re + 1)$-th minors of $\mathcal{E}$, and let
  \begin{equation*}
    Z_r  = \bigcap_{\mathcal{E}} Z_{r, \mathcal{E}}\text{.}
  \end{equation*}
  Suppose either
  \begin{itemize}
    \item $n \geq 6$ and $r \geq 14$ or
    \item $n = 5$ and $r \geq 42$ or
    \item $n = 4$ and $r \geq 140$.
  \end{itemize}
  Suppose $d \geq 2r -1$. Then $\sigma_r(v_d(\mathbb{P}V)) \subsetneq Z_r$.
\end{corollary}
\begin{proof}
  From \cite[Proposition 6.2 and Theorem 1.4]{nisiabu_jabu_cactus} for such $n, d, r$ we have $\sigma_r(v_d(\mathbb{P}V)) \subsetneq
  \kappa_r(v_d(\mathbb{P}V))$. Since from Theorem \ref{theorem:vector_bundle_bound} we know that $\kappa_r(v_d(\mathbb{P}V)) \subseteq Z_r$, we get
  the desired result.
\end{proof}

In Section \ref{section:homogeneous_and_reflexive} we pose some questions about the other version of the Landsberg-Ottaviani method and some
generalizations.

\subsection{Acknowledgements}
I would like to thank Jaros{\l}aw Buczy{\'n}ski for drawing my attention to the problem of equations given by vector bundles and many helpful
discussions. I also thank J. M. Landsberg and Giorgio Ottaviani for suggesting to explore the homogeneous version of the method. I was supported by
Warsaw Center of Mathematics and Computer Science financed by Polish program KNOW.

\section{Proof of Theorem \ref{theorem:vector_bundle_bound}}\label{section:proof}

For a closed subscheme $i\colon R \hookrightarrow X$, $\langle R \rangle$ denotes its linear span in $\mathbb{P}(H^0(X, \mathcal{L})^*)$, and
$\mathcal{I}_R$ denotes its ideal sheaf on $X$. Recall that for any locally free sheaf $\mathcal{E}$ of finite rank on $X$, the vector subspace
$H^0(X, \mathcal{I}_R \otimes \mathcal{E}) \subseteq H^0(X, \mathcal{E})$ consists of the sections which pull back to zero on $R$.

Let $(\cdot \hook \cdot): H^0(X, \mathcal{L}) \otimes H^0(X, \mathcal{L})^* \to \mathbb{K}$ denote the natural pairing. 

\begin{proposition}[Apolarity lemma, general version]
  \label{proposition:general_apolarity}
  Let $F \in H^0(X, \mathcal{L})^*$ be a non-zero element. Then for any closed subscheme $i : R \hookrightarrow X$ we have
  \begin{equation*}
    F \in \langle R \rangle \iff H^0(X, \mathcal{I}_R \otimes \mathcal{L}) \hook F = 0\text{.}
  \end{equation*}
\end{proposition}
Proposition \ref{proposition:general_apolarity} was proven independently in \cite[Proposition 3.7]{mgalazka_multigraded_apolarity} and \cite[Lemma
1.3]{toric_ranestad}. We include a proof here for the reader's convenience.
\begin{proof}
  Take any $s \in H^0(X, \mathcal{L})$, let $H_s$ be the corresponding hyperplane in $H^0(X, \mathcal{L})^*$. Then,
  \begin{equation*}
    \langle R \rangle \subseteq H_s \iff i^*(s) = 0 \iff s \in H^0(X, \mathcal{I}_R \otimes \mathcal{L}) \text{.}
  \end{equation*}
  Below we identify sections $s \in H^0(X, \mathcal{L})$ with hyperplanes $H_s$ in $H^0(X, \mathcal{L})^*$. Then for any $R$
  \begin{align*}
    F \in \langle R \rangle & \iff \forall_{s\in H^0(X, \mathcal{L})}(\langle R \rangle \subseteq H_s \implies F \in H_s) \\
    & \iff \forall_{s\in H^0(X, \mathcal{L})}(s \in H^0(X, \mathcal{I}_R \otimes \mathcal{L}) \implies F \in H_s) \\
    & \iff \forall_{s\in H^0(X, \mathcal{L})}(s \in H^0(X, \mathcal{I}_R \otimes \mathcal{L}) \implies s \hook F = 0) \\
    & \iff H^0(X, \mathcal{I}_R \otimes \mathcal{L}) \hook F = 0 \text{.}
  \end{align*}
\end{proof}
Take a vector bundle (i.e.\ a locally free sheaf) $\mathcal{E}$ on $X$. Let $\hook$ be the map
\begin{equation*}
  H^0(X, \mathcal{E}) \otimes H^0(X, \mathcal{L})^* \xra{\cdot\hook\cdot} H^0(X, \mathcal{L}\otimes \mathcal{E}^\vee)^*
\end{equation*}
given by rearranging terms of Equation \eqref{equation:j_apolarity}. The notation agrees with the one in Proposition
\ref{proposition:general_apolarity} when we set $\mathcal{E} = \mathcal{L}$. Then if we fix $F \in H^0(X, \mathcal{L})^*$, we get the so-called
catalecticant homomorphism $C_F^\mathcal{E} = \cdot \hook F$
\begin{equation*}
  C_F^\mathcal{E} : H^0(X, \mathcal{E}) \to H^0(X, \mathcal{L} \otimes \mathcal{E}^\vee)^*
\end{equation*}
corresponding to the tensor
\begin{equation*}
  j^\mathcal{E}(F) \in H^0(X, \mathcal{E})^* \otimes H^0(X, \mathcal{L} \otimes \mathcal{E}^\vee)^*\text{.}
\end{equation*}
Notice that $\rank C_F^\mathcal{E} = \rank j^\mathcal{E}(F)$.
\begin{proposition}\label{proposition:nonabelian_apolarity}
  Suppose $R \hookrightarrow X$ is a subscheme of $X$. Now, if $F \in \langle R \rangle$, then
  \begin{equation*}
    H^0(X, \mathcal{I}_R \otimes \mathcal{E}) \subseteq \ker C_F^\mathcal{E}\text{.}
  \end{equation*}
\end{proposition}
\begin{remark}
  Proposition \ref{proposition:nonabelian_apolarity} can be thought of as a version of the apolarity lemma for vector bundles (or even any sheaves since
  the proof works also in this case).
\end{remark}
\begin{proof}[Proof of Proposition \ref{proposition:nonabelian_apolarity}]
  We have the following commutative diagram of sheaves on $X$:
  \[\begin{tikzcd}
      \mathcal{I}_R \otimes \mathcal{E}\otimes \mathcal{E}^\vee \otimes \mathcal{L} \arrow{d}\arrow{r} & \mathcal{I}_R\otimes \mathcal{L}\arrow{d} \\
      \mathcal{E} \otimes \mathcal{E}^\vee \otimes \mathcal{L} \arrow{r} & \mathcal{L}\text{.}
  \end{tikzcd}\]
  This gives rise to a commutative diagram of global sections:
  \[\begin{tikzcd}
      H^0(X, \mathcal{I}_R\otimes \mathcal{E})\otimes H^0(X, \mathcal{E}^\vee \otimes \mathcal{L})\arrow[dashed]{d}\arrow{r} & H^0(X,\mathcal{I}_R \otimes
      \mathcal{L})\arrow{d}\arrow[bend left, dotted]{ddr} \\
      H^0(X, \mathcal{E})\otimes H^0(X,\mathcal{E}^\vee \otimes \mathcal{L})\arrow[bend right, dashed]{rrd}\arrow{r} & H^0(X, \mathcal{L})
      \arrow{rd}[swap]{\cdot \hook F} \\
      & & \mathbb{K} \text{.}
  \end{tikzcd}\]
  We want to show that the composition along the dashed path is zero. It suffices to show that the dotted arrow is zero. But it is true by
  Proposition~\ref{proposition:general_apolarity}.
\end{proof}
\begin{proposition}\label{proposition:zero_dimensional}
  Let $R$ be a zero-dimensional subscheme of $X$ with ideal sheaf $\mathcal{I}_R$.  Then for any locally free sheaf $\mathcal{E}$ of finite rank $e$, the
  length of $R$ is greater or equal to
  \begin{equation*}
    \frac{h^0(X, \mathcal{E}) - h^0(X, \mathcal{I}_R \otimes \mathcal{E})}{e}\text{.}
  \end{equation*}
\end{proposition}
\begin{proof}
  We have an exact sequence
  \begin{equation*}
    0 \to \mathcal{I}_R \to \mathcal{O}_X \to \mathcal{O}_R \to 0\text{.}
  \end{equation*}
  We tensor it with $\mathcal{E}$:
  \begin{equation*}
    0 \to \mathcal{I}_R\otimes \mathcal{E} \to \mathcal{E} \to \mathcal{E}_{|R} \to 0\text{.}
  \end{equation*}
  After taking global sections (which are left-exact), we get an exact sequence
  \begin{equation*}
    0 \to H^0(X, \mathcal{I}_R\otimes \mathcal{E}) \to H^0(X, \mathcal{E}) \to H^0(R, \mathcal{E}_{|R})\text{.}
  \end{equation*}
  It follows that 
  \begin{equation*}
    h^0(R, \mathcal{E}_{|R}) \geq h^0(X,\mathcal{E}) - h^0(X, \mathcal{I}_R \otimes \mathcal{E})\text{.}
  \end{equation*}
  But on a zero-dimensional scheme, every locally free sheaf trivializes. This means $h^0(R, \mathcal{E}_{|R}) = h^0(R, \mathcal{O}^{\oplus e}_R)$, which is
  equal to $e$ times the length of $R$.
\end{proof}
\begin{proof}[Proof of Theorem \ref{theorem:vector_bundle_bound}]
  Take any zero-dimensional scheme $R\hookrightarrow X$ such that $F \in \langle R \rangle$. We have
  \begin{align*}
    \length R \geq \frac{h^0(X, \mathcal{E}) - h^0(X, \mathcal{I}_R \otimes \mathcal{E})}{e} \geq \frac{h^0(X, \mathcal{E}) - \dim \ker
    C_F^\mathcal{E}}{e} &= \frac{\dim \im C_F^\mathcal{E}}{e}  \\
    &= \frac {\rank j(F)} e\text{,}
  \end{align*}
  where the first inequality follows from Proposition \ref{proposition:zero_dimensional}, and the second from Proposition
  \ref{proposition:nonabelian_apolarity}.
\end{proof} 
\section{Homogeneous and reflexive analogues}\label{section:homogeneous_and_reflexive}
There is another version of the Landsberg-Ottaviani method. This version is exploited in \cite[Section 4]{landsberg_ottaviani_equations} in the case
of Veronese embedding.

For this section assume $\mathbb{K} = \mathbb{C}$. Let $G$ be a complex semisimple group, let $V$ be an irreducible $G$-module, and $X \subseteq
\mathbb{P}V$ the unique closed orbit (so that $X$ is a $G$-homogeneous variety). Suppose we have an inclusion of $G$-modules $V \xra{j} W_1\otimes
W_2$, where $W_1$ and $W_2$ are two irreducible $G$-modules. Suppose $\rank j(F) = k$ for any $[F] \in X$ (it is the same for any $[F] \in X$ since
$X$ is $G$-homogeneous).
Then
\begin{proposition}
  \label{proposition:landsberg_ottaviani_modules}
  For any $F$ we have
  \begin{equation*}
    \rr(F) \geq \frac{\rank j(F)}{k}\text{.}
  \end{equation*}
  In other words, if we fix bases of $W_1$ and $W_2$, then for any $r\geq 0$ the $(rk+1)$-th minors
  of $j$ (which is a matrix with entries in $V^*$) give equations for $\sigma_r(X)$.
\end{proposition}

\begin{question}
  Does the inequality in Proposition \ref{proposition:landsberg_ottaviani_modules} hold for the cactus rank?
\end{question}

In \cite[Proposition 1.7]{mgalazka_multigraded_apolarity} the author proved the following proposition.
\begin{proposition}
  If $X \subseteq \mathbb{P}W$ is a projective toric variety embedded by the complete linear system of a very ample line bundle, then reflexive sheaves
  $\mathcal{F}$ of rank one give equations for the secant varieties, i.e.\ we have
  \begin{equation*}
    \rr(F) \geq \rank j^\mathcal{F}(F)
  \end{equation*}
  for every $F \in W$. Here $j^{\mathcal{F}}$ is defined in the same way as in Construction \ref{construction:vector_bundle_version}, but for a
  reflexive sheaf instead of a vector bundle.
\end{proposition}
There are examples of reflexive sheaves of rank one which do not give equations of the cactus variety, i.e.\ $\crr(F) \ngeq \rank j^{\mathcal{F}}(F)$
for some $F \in W$ (see \cite[Subsection 5.2]{mgalazka_multigraded_apolarity}). This gives rise to the following question:
\begin{question}
  Suppose $X \subseteq \mathbb{P} W$ is embedded by the complete linear system of a very ample line bundle, and $\mathcal{F}$ is a reflexive sheaf of
  rank $f$ on $X$. Take any $F \in W$. Is it true that
  \begin{equation*}
    \rr(F) \geq \frac {\rank j^{\mathcal{F}}(F)} f \text{?}
  \end{equation*}
\end{question}

There may be other kinds of sheaves which give equations of secant varieties or cactus varieties in a similar way.

\end{document}